\documentclass{article}

\usepackage[utf8]{inputenc}
\usepackage[english]{babel}
\usepackage{geometry}
\usepackage{microtype}
\usepackage{amsmath}
\usepackage{amssymb}
\usepackage{amsthm}
\usepackage{graphicx} % Required for inserting images
\usepackage{subcaption}
\usepackage{cite}
\usepackage{xcolor}
\usepackage{thmtools}
\usepackage{thm-restate}

\RequirePackage[colorlinks,citecolor=blue,urlcolor=blue]{hyperref}
\usepackage{cleveref}

\newcommand{\oeis}[1]{\href{http://oeis.org/#1}{#1}}
\newcommand{\addorcid}[1]{\href{https://orcid.org/#1}{\protect\includegraphics[height=3mm]{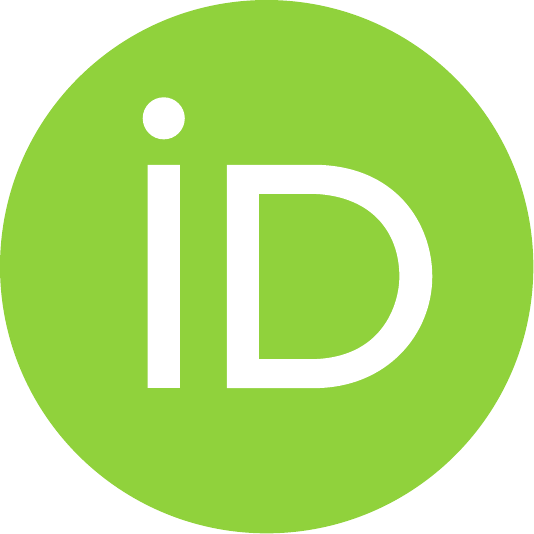}} }

\definecolor{catred}{RGB}{237, 41, 57}
\definecolor{lightblue}{HTML}{A5D8FF}
\newtheorem{theorem}{Theorem}
\newtheorem{lemma}[theorem]{Lemma}
\newtheorem{corollary}[theorem]{Corollary}
\newtheorem{definition}[theorem]{Definition}

\theoremstyle{definition}

\newtheorem{remark}[theorem]{Remark}

\newcommand{\Sc}{\mathcal{S}}
\newcommand{\Z}{\mathbb{Z}}

\newcommand{\ECat}{{\color{catred}\textbf{E}}}

\title{A Bijection Between Stacked Directed Polyominoes \\ and Motzkin Paths with Catastrophes}

\author{
Florian Schager\addorcid{0009-0009-3923-051X}
\and 
Michael Wallner\addorcid{0000-0001-8581-449X} }

\begin{document}
\maketitle

\begin{abstract}
We present a novel bijection between stacked directed polyominoes and Motzkin paths with catastrophes. Further, we leverage this new bridge between these two worlds to obtain a better understanding of certain parameters of stacked directed animals. In particular, we obtain improved lower and upper bounds on the asymptotic width of stacked directed animals.

\medskip

\end{abstract}

\tableofcontents

\pagebreak

\section{Introduction}

In this paper we show a bijection between a subclass of polyominoes called \textit{stacked directed animals} and a class of Motzkin paths augmented with a model of catastrophes. The motivation behind the enumeration of lattice animals or polyominoes can be found in the study of branched polymers \cite{Polymers} and percolation \cite{Percolation}. 
Even though these combinatorial objects have been studied for more than 40 years, exact enumeration results for general polyominoes are still rare. 
Thus, one of the main research directions focuses on the investigation of large subclasses of polyominoes that are exactly enumerable. 
This is also the motivating force behind the introduction of the class of stacked directed animals.

\subsection{Lattice animals}

\begin{definition}[Lattice animals]\label{def:lattice_animals}
  A \emph{polyomino} of area $n$ is a connected union of $n$ cells on a lattice. 
  The corresponding \emph{lattice animal} then lives on the dual lattice obtained by taking the center of each cell.
\end{definition}

Stacked directed animals were first introduced in a paper by Bousquet-Mélou and Rechnitzer in 2002 \cite{LatticeAnimals} as a superset of the well-studied class of \textit{directed animals}.
Informally speaking, stacked directed animals can be seen as a sequence of directed animals.

\begin{definition}[Directed animals]
\label{def:directed_animal}
  A \emph{directed animal} on the square grid is a lattice animal, where one vertex has been designated the \emph{source} and all other vertices are connected to the source via a directed path consisting only of north (\textbf{N}) and east (\textbf{E}) steps, and visiting only vertices belonging to the animal.
\end{definition}

The enumeration of directed animals is already well mapped out. 
The first exact enumeration results are due to Dhar \cite{Dhar82, Dhar83} who showed that the number of directed animals of size $n$ grows asymptotically as $3^n$.
The derivation of the generating function for directed animals has been greatly simplified by the introduction of a key auxiliary concept.

\begin{definition}[Heaps of dimers]
  A \emph{dimer} consists of two adjacent vertices on a lattice. 
  A \emph{heap of dimers} is obtained by dropping a finite number of dimers towards a horizontal axis, where each dimer falls until it either touches the horizontal axis or another dimer; see Figure~\ref{fig:heaps}. 
  The \emph{width} of a heap is the number of non-empty columns. 
  The dimers that touch the $x$-axis are called \emph{minimal}.
  A heap is called
  \begin{itemize}
    \item \emph{strict}, if no dimer has another dimer directly above it;
    \item \emph{connected}, if its orthogonal projection on the horizontal axis is connected;
    \item a \emph{pyramid}, if it has only one minimal dimer;
    \item a \emph{half-pyramid}, if its only minimal dimer lies in the rightmost non-empty column.
  \end{itemize}
  The \emph{right/left width} of a pyramid is the number of non-empty columns to the right/left of the minimal dimer.
\end{definition}

This combinatorial structure was first introduced by Viennot in 1985 \cite{Heaps}.
In addition to simplifying the derivation of the generating functions of directed animals, they also serve as an intermediary step for the bijection we present in this paper.

\begin{figure}[hbt!]
    \newcommand{\mywidth}{0.2 \textwidth}
    \newcommand{\myheight}{2.8cm}
  \centering
  \begin{subfigure}{\mywidth}
    \centering
    \includegraphics[height=\myheight]{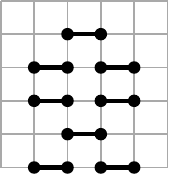}
    \caption{A general heap.}
  \end{subfigure}
  \hfill
  \begin{subfigure}{\mywidth}
    \centering
    \includegraphics[height=\myheight]{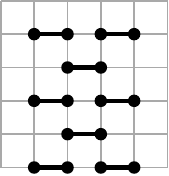}
    \caption{A strict heap.}
  \end{subfigure}
  \hfill
  \begin{subfigure}{\mywidth}
    \centering
    \includegraphics[height=\myheight]{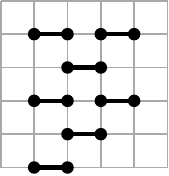}
    \caption{A pyramid.}
  \end{subfigure}
  \hfill
  \begin{subfigure}{\mywidth}
    \centering
    \includegraphics[height=\myheight]{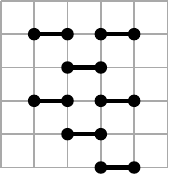}
    \caption{A half-pyramid.}
  \end{subfigure}
  \caption{Different types of heaps of dimers.}
  \label{fig:heaps}
\end{figure}

\subsection{Lattice paths with catastrophes}

On the other side of our bijection we have lattice paths.

\begin{definition}[Lattice paths]
Let $\Sc \subseteq \Z$ be a finite set of integers called \emph{steps}. 
A \emph{lattice path} is a sequence $(s_1,s_2,\dots,s_n) \in \Sc^n$ of steps with a fixed starting point $y_0$.\footnote{In the literature this model is called simple and directed.}
We set $y_0=0$ and define $y_k = \sum_{i=1}^k s_i$ as the \emph{altitude} of the path after $k$ steps.
Furthermore, we distinguish different classes of paths:
\begin{itemize}
    \item An unconstrained path is called a \emph{walk}.
    \item A walk ending on the $x$-axis (i.e., $y_n=0$) is called a \emph{bridge}.
    \item A walk that may never cross the $x$-axis (i.e., $y_k \geq 0$) is called a \emph{meander}.
    \item A walk that is at the same time a bridge and a meander is called an \emph{excursion}.
\end{itemize}    
\end{definition}

\emph{Dyck paths} are probably the most ubiquitous class of paths, which are excursions associated with the steps $\Sc = \{-1,1\}$ and famously enumerated by the Catalan numbers.
In our bijections we will encounter the nearly equally well-known \emph{Motzkin paths}, which are excursions associated with $\Sc = \{-1,0,1\}$. In writing we will identify the steps with their cardinal directions (\textbf{SE}, \textbf{E}, \textbf{NE}) when drawn on the cartesian plane.
We will call them Motzkin excursions (resp.\ Motzkin meanders), when they use the step set of Motzkin paths and are excursions (resp.\ meanders).
For a simple step set $\mathcal{S}$ we define the \textit{characteristic polynomial} as $P(u) = \sum_{s \in \mathcal{S}} u^s$.
To derive generating functions for various classes of lattice paths, one often looks at solutions of the \textit{kernel equation} $K(z,u) := 1 - z \cdot P(u) = 0$. When $c = - \min(\mathcal{S})$ and $d = \max(\mathcal{S})$, the solutions can be divided into $c$ \textit{small roots} $u_1(z), \dots, u_c(z)$ with $u_i(z) \xrightarrow[z \to 0]{} 0$ and $d$ \textit{large roots} $v_1(z), \dots, v_d(z)$ with $v_j(z) \xrightarrow[z \to 0]{} \infty$.
Then, the following classic result from Banderier and Flajolet gives an explicit expression for the generating function of lattice paths with arbitrary simple step sets.

\begin{theorem}[Generating function of meanders and excursions {\cite[Theorem 2]{BasicLatticePaths}}]
\label{thm:gf_meanders_excursions}
  The bivariate generating function of meanders ($z$ marking size and $u$ marking final altitude) relative to a simple step set $\mathcal{S}$ with characteristic polynomial $P(u)$, is an algebraic function. It is given by
  \begin{equation} \label{eq:gf_meanders}
    M(z,u) = \frac{\prod_{j=1}^c (u - u_j(z))}{u^c(1-zP(u))} = -\frac{1}{p_dz} \prod_{\ell = 1}^d \frac{1}{u - v_\ell(z)}.
  \end{equation}
  In particular, the generating function of excursions, $E(z) = M(z,0)$ satisfies
  \begin{equation} \label{eq:gf_excursions}
    E(z) = \frac{(-1)^{c-1}}{p_{-c}z}
    \prod_{j=1}^c u_j(z) = \frac{(-1)^{d-1}}{p_dz} \prod_{\ell=1}^d \frac{1}{v_\ell(z)}.
  \end{equation}
\end{theorem}

We apply the theorem to get an explicit expression for the generating functions of Motzkin meanders and excursions, which we will need later on in the paper.

\begin{corollary}[Motzkin meanders] \label{cor:motzkin_meanders}
  The bivariate generating function $M_{\mathcal{M}}(z,u)$ for Motzkin meanders satisfies 
  $$
    M_{\mathcal{M}}(z,u) = \frac{2z(u + 1) - 1 + \sqrt{1 - 2z - 3 z^{2}}}{2z\left(u - z \left(u^{2}+u + 1\right)\right)}.
  $$
  Further, the generating function $M_{\mathcal{M}}(z, 1)$ of meanders ending at any altitude satisfies
  \begin{align*}
    M_{\mathcal{M}}(z, 1) &= \frac{1 - 3z - \sqrt{1 - 2z - 3z^2}}{6z^2 - 2z} \\
    &= 1 + 2z + 5z^{2} + 13z^{3} + 35z^{3} + 96z^{5} + 267z^{6} + 750z^7 + \mathcal{O}(z^8).
  \end{align*}
  This counting sequence corresponds to \href{https://oeis.org/A005773}{\texttt{OEIS A005773}}, which tells us that it also counts the number of directed animals of size $n$. We will make this connection explicit in \Cref{sec:bijection}.
\end{corollary}

\begin{proof}
  Solving the kernel equation
  $$
    1 - z(1 + u + u^2) = 0
  $$
  for the Motzkin family of directed lattice paths yields the unique small branch
  $$
    u_1(z) = \frac{1 - z - \sqrt{1 - 2z - 3z^2}}{2z}.
  $$
  Now, all that remains is to plug $u_1(z)$ into Equation \eqref{eq:gf_meanders} and we obtain
  $$
    M_{\mathcal{M}}(z,u) = \frac{u - u_1(z)}{u(1 - zP(u))} = \frac{2z(u + 1) - 1 + \sqrt{1 - 2z - 3 z^{2}}}{2z\left(u - z \left(u^{2}+u + 1\right)\right)}.
  $$
  Setting $u = 1$ then yields
  \begin{equation*}
    M_{\mathcal{M}}(z,1) = \frac{1 - 3z - \sqrt{1 - 2z - 3z^2}}{6z^2 - 2z}. \qedhere
  \end{equation*}
\end{proof}

\begin{corollary}[Motzkin excursions]
  \label{cor:motzkin_excursions}
  The generating function $E_{\mathcal{M}}(z,1)$ for Motzkin excursions satisfies 
  \begin{align*}
    E_{\mathcal{M}}(z,1) &= \frac{u_1(z)}{z} = \frac{1 - z - \sqrt{1 - 2z - 3z^2}}{2z^2} \\
    &= 1 + z + 2z^{2} + 4z^{3} + 9z^{4} + 21z^{5} + 51z^{6} + 127z^{7} + \mathcal{O}(z^{8})
  \end{align*}
  and its coefficients are known as the \textit{Motzkin numbers}; see \href{https://oeis.org/A001006}{\texttt{OEIS A001006}}
\end{corollary}
    
We will now enrich these models by a new type of step that takes a path immediately down to the $x$-axis; see \Cref{fig:catastrophes_example}

\begin{definition}[Catastrophe]
  \label{def:catastrophe} 
  For $s \geq 0$, a \emph{catastrophe} is a step $-s$ that is only allowed to happen at altitude~$s$.
\end{definition}

Dyck meanders with catastrophes were first introduced in 2005 by Krinik et al.~\cite{QueueingTheory} as a model for the classical single server queueing system M/M/1/H with a finite capacity, with a constant catastrophe rate $\gamma$. 
In addition, catastrophe queues also arise as simple, natural models of the evolution of stock markets \cite{Schoutens}, or under the name of \emph{random walks with resetting} in the field of probability theory and statistical mechanics \cite{RandomWalksResetting}.
Later, Banderier and Wallner~\cite{Catastrophes} studied enumerative properties and derived limit laws for several parameters, such as the total number of catastrophes.
In their paper, they actually use a slightly different model of catastrophes, where a catastrophe is not allowed to coincide with a regular down step.
Removing this restriction leads to a model that is easier to handle and simplifies the derivation of generating functions as well as the study of parameters.
Furthermore, it is a necessary adaptation in order for our bijection to go through.

\begin{figure}[hbt!]
    \centering
    \includegraphics[height=3cm]{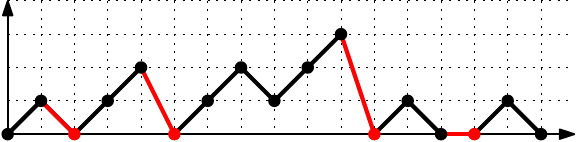}
    \caption{Example of a Dyck excursion with catastrophes marked in red. 
    Note that a step from altitude $1$ to $0$ can either be a catastrophe (marked in red) or a step $-1$ from the step set $\mathcal{S} = \{ -1, 1 \} $. }
    \label{fig:catastrophes_example}
\end{figure}

\subsection{Construction of stacked directed animals}
\label{subsec:construction}

Now we will describe a construction from \cite[p.~240]{LatticeAnimals} that maps directed animals on the square lattice to strict pyramids of dimers.

\begin{definition}[Mapping from directed animals to heaps]
  Let $\mathcal{D}$ denote the set of directed lattice animals on the square grid, $\mathcal{P}$ denote the set of strict pyramids and $D \in \mathcal{D}$. 
  We define a mapping $V \colon \mathcal{D} \to \mathcal{P}$ as follows:
  \begin{enumerate}
    \item Rotate $D$ by $45^\circ$ degrees counter-clockwise.
    \item Replace each individual cell in $D$ by a dimer.
  \end{enumerate}
  This results in a pyramid that we call $V(D)$, with the source of the lattice animal being the only minimal dimer.
\end{definition}

\begin{remark} \label{remark:bijection}
  It was observed by Viennot in \cite{Heaps} that this mapping induces a bijection between directed animals on the square lattice and strict pyramids of dimers and we denote the inverse mapping by $\overline{V}$.
  This can be easily verified by recalling that any vertex in $D$ lies on a directed path consisting only of \textbf{N} and \textbf{E} steps from the source, visiting only other vertices in $D$. 
  Hence, the corresponding dimer in $V(D)$ lies on a directed path of dimers lying diagonally to the left or the right above each other. 
  As the next definition will show, it only takes a small adaptation to extend this mapping to general lattice animals.
\end{remark}

\begin{figure}[hbt!]
  \centering
  \includegraphics[height=3.6cm]{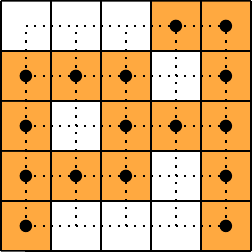}
  \qquad \qquad 
    \includegraphics[height=4.5cm]{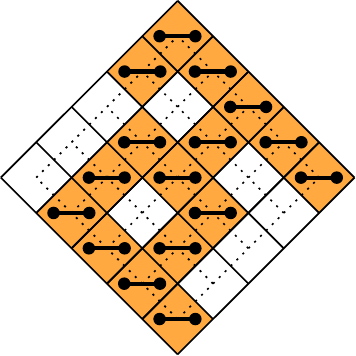}
  \qquad \qquad 
    \includegraphics[height=3.6cm]{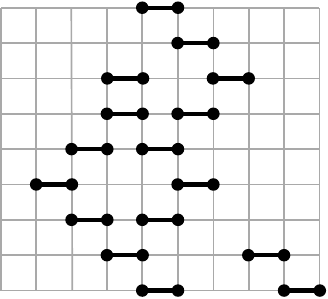}
  \caption[Constructing a connected heap from a square lattice animal.]{Constructing the connected heap $V(A)$ from an animal $A$ on the square grid.}
  \label{fig:connected_heap}
\end{figure}

\begin{definition}[Mapping from lattice animals to heaps]
  Let $\mathcal{A}$ denote the set of lattice animals on the square lattice, $\mathcal{H}$ the set of connected heaps, and
  $A \in \mathcal{A}$.
  We define a mapping $V: \mathcal{A} \to \mathcal{H}$ as follows:
  \begin{enumerate}
    \item Rotate $A$ by $45^\circ$ degrees counter-clockwise.
    \item Replace each individual cell in $A$ by a dimer.
    \item Let the dimers fall.
  \end{enumerate}
  We call the resulting heap $V(A)$; see Figure \ref{fig:connected_heap} for an example of this procedure.
\end{definition}

\smallskip

Thus, $V$ maps square lattice animals to connected heaps that are not necessarily strict. However, clearly not every connected heap can be obtained in this way. Hence, we restrict our attention to strict, connected heaps and define a class of lattice animals that stand in one-to-one correspondence with strict, connected heaps via $V$.

\smallskip

\begin{definition}[Multi-directed animals]
\label{def:multi_directed_animals}
  Let $H$ be a strict, connected heap. 
  We now construct an extension of $\overline{V}$ to connected heaps via induction over the number of minimal dimers $k$ of $H$:
  \begin{itemize}
    \item For $k = 1$, the heap $H$ reduces to a simple pyramid. Thus, by Remark~\ref{remark:bijection}, $\overline{V}(H)$ is already well-defined.
    \item If instead $H$ has $k > 1$ minimal dimers, we push the $(k-1)$ leftmost pyramids upwards, producing a connected heap $H'$ with $k-1$ minimal dimers, placed far above the remaining pyramid $P_k$. 
    Now, recursively replace $H'$ by $\overline{V}(H')$ and $P_k$ by $\overline{V}(P_k)$.
    \item Finalize the construction by pushing $\overline{V}(H')$ downwards until it connects to $\overline{V}(P_k)$.
  \end{itemize}
  We define $\overline{V}(H)$ as the resulting animal and call the class of square lattice animals obtainable in this way \emph{square multi-directed animals}.
\end{definition}

\begin{figure}[t]
    \centering
    \includegraphics{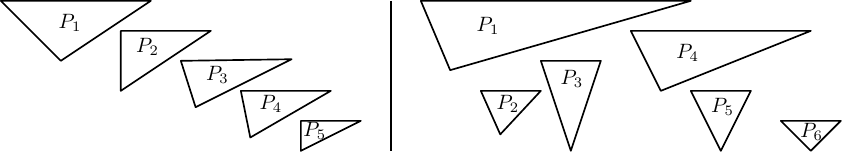}
    \caption{The schematic structure of stacked directed animals (left) and multi-directed animals (right) from Definitions~\ref{def:stacked_directed_animals} and \ref{def:multi_directed_animals}, respectively \cite[Figure~9]{LatticeAnimals}. Each triangle represents a directed animal from Definition~\ref{def:directed_animal}. }
    \label{fig:multi_directed}
\end{figure}

\smallskip

Now we are ready to define stacked directed animals as a subclass of multi-directed animals.

\begin{definition}[Stacked directed animals]
\label{def:stacked_directed_animals}
  Take a connected heap $H$ with $k$ minimal dimers. Let us denote by $P_1,P_2,\dots,P_k$, from left to right, the corresponding \emph{pyramidal factors} of $H$ from the construction in Definition~\ref{def:multi_directed_animals}. 
  Let us call \emph{stacked pyramids} the connected heaps such that for $2 \leq i \leq k$, the horizontal projection of $P_i$ intersects the horizontal projection of $P_{i-1}$. 
  Then, \emph{stacked directed animals} are defined as the image of the set of stacked pyramids under $\overline{V}$.
  The \emph{right width} of a stacked pyramid is the right width of its rightmost pyramidal factor. 
\end{definition}

\smallskip

These lattice animals are easier to enumerate due to their recursive description, visualized in Figure~\ref{fig:multi_directed}.
This description yields algebraic equations for their generating functions.
It will also prove crucial in constructing our correspondence to Motzkin excursions with catastrophes that we introduce in the next section.

\subsection{Our contributions}

As our main result, we construct the following bijection:

\begin{restatable}{theorem}{Bijection} 
  \label{thm:bijection}
  The set of Motzkin excursions with catastrophes of length~$n$ is in bijection with the set of stacked directed animals of size $n + 1$ on the square grid.
\end{restatable}

For an visual proof, we refer to \Cref{fig:stacked_pyramids}, which shows that the two combinatorial classes admit equivalent structural decompositions. 
These decompositions then directly yield a procedure to transform a stacked directed animal into a Motzkin excursion with catastrophes and vice versa:
Starting with an arbitrary stacked directed animal, one first notes down the recursive decomposition of the animal, translates the bases cases into the corresponding lattice paths and traces the steps of the decomposition back on the equivalent lattice path side in reverse order.

As an application of our bijection, we give improved upper and lower bounds on the asymptotic width of stacked directed animals. The best previously known lower bound was $\frac{3}{28} n$ \cite[Proposition 2]{LatticeAnimals}, while on the upper bound side no non-trivial bound was known beforehand.

\begin{restatable}{theorem}{Width}
    \label{thm:width}
    We define the random variable $W_n$ as the width of a stacked pyramid of size $n$ drawn uniformly at random. Then, $ \lim_{ n \to \infty } \mathbb{E}[W_n] \in \left[ \frac{9}{28n}, \frac{6n}{7}\right]$.
\end{restatable}

The remainder of the paper is organized as follows: In \Cref{sec:generating_functions} we show that both classes are described by the same generating function, thus proving that they are equinumerous. Then, in \Cref{sec:bijection} we incrementally build up to the final bijection. As a first stepping stone we prove that half-pyramids are in bijection with Motzkin excursions. Next, we move onto pyramids, which correspond to Motzkin excursions with catastrophes, but only at height $h = 0$. Finally, using the previous two bijections as core building block we can prove our main result, \Cref{thm:bijection}.
In \Cref{sec:parameters} we study the average width of a stacked pyramids. Using our bijection, we can translate results on the distribution of certain statistics of lattice paths to the polyomino setting, thus proving \Cref{thm:width}.
The cumulative size of all catastrophes on the lattice path side gives the improved lower bound of $\frac{9}{28} n$ on the average width of a stacked pyramid.
On the other hand, we show that counting all steps except the \textbf{D} steps gives an upper bound on the width of stacked pyramids of size $n$ of $\frac{6n}{7}$.

\begin{figure}[hbt!]
  \centering
  \includegraphics[width = \linewidth]{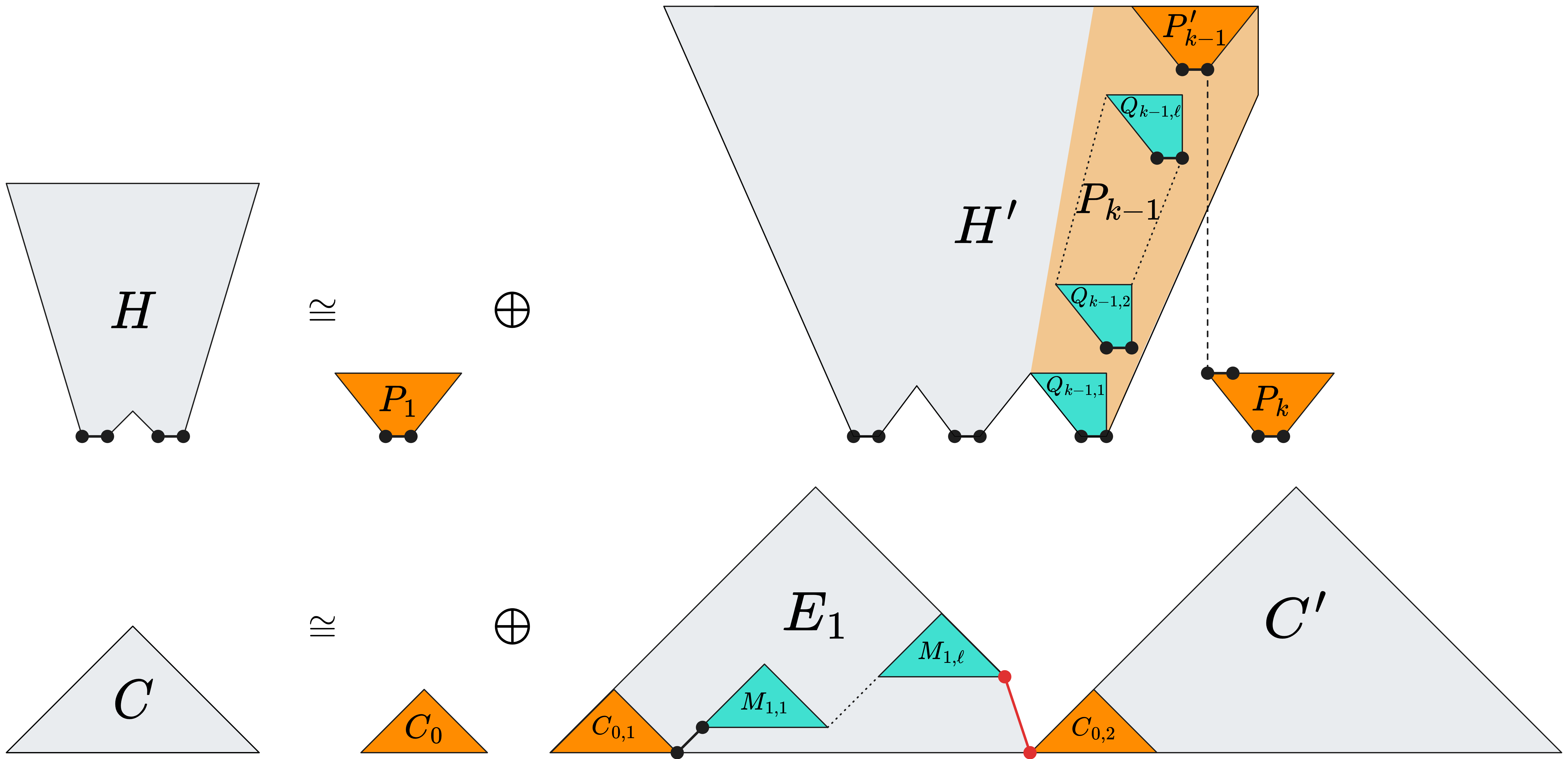}
  \caption[Recursive construction of stacked pyramids.]{The recursive constructions of stacked pyramids and Motzkin excursions with catastrophes.}
  \label{fig:stacked_pyramids}
\end{figure}

\section{Equality of generating functions}
\label{sec:generating_functions}

In this section we show that the class of Motzkin excursion with catastrophes and the class of stacked directed animals are described by the same generating function and are thus equinumerous.

\begin{theorem}[Generating function for meanders and excursions with catastrophes {\cite[Theorem 2.1]{Catastrophes}}] \label{thm:gf_catastrophes}
  Let $c_{n,k}$ be the number of meanders with catastrophes of length $n$ ending at altitude $k$, relative to a simple step set $\mathcal{S}$, with characteristic polynomial $P(u) = \sum_{k=-c}^d p_k u^k$.
  Further, let $u_1,\dots,u_c$ denote the small roots and $v_1,\dots,v_d$ the large roots of the kernel equation.
  Then the generating function 
  $$
  C(z,u) = \sum_{n,k = 0}^\infty c_{n,k} u^k z^n
  $$ 
  is algebraic and satisfies 
  \begin{equation*}
    C(z,u) = D(z) \cdot M(z,u) = \frac{1}{1 - Q(z)} \cdot \frac{\prod_{i=1}^c(u - u_i(z))}{u^c(1 - zP(u))},
  \end{equation*}
  where $D(z)$ denotes the generating function of excursions ending with a catastrophe and $Q(z) = z \cdot M(z, 1)$ counts the number of excursions with exactly one catastrophe occurring as the last step of the path. In particular, the generating function $E(z)$ for excursions with catastrophes satisfies
  \[
    E(z) = D(z) \cdot M(z,0) = \frac{(-1)^{c-1}}{1 - Q(z)} \cdot \frac{\prod_{i=1}^c u_i(z)}{z p_{-c}}
  \]
\end{theorem}

We recall Motzkin walks to be directed lattice paths with the simple step set $\mathcal{M} = \{-1,0,1\}$.

\begin{corollary} \label{ex:motzkin_excursions}
  The generating function $E_\mathcal{M}(z)$ of Motzkin excursions with catastrophes satisfies 
  $$
    E_\mathcal{M}(z) = D(z) \cdot M_{\mathcal{M}}(z, 0) = \frac{E(z)}{1 - Q(z)} = \frac{u_1(z)}{z\left(1 - z\frac{1 - u_1(z)}{1 - 3z}\right)} = \frac{u_{1}(z)(1 - 3z)}{z(1 + (u_{1}(z) - 4)z)}.
  $$
  Extracting the first few coefficients yields 
  $$ 
    E_\mathcal{M}(z) = 1 + 2z + 6z^{2} + 19z^{3} + 63z^{4} + 213z^{5} + 729z^{6} + 2513z^{7} + \mathcal{O}(z^{8}).
  $$
  This sequence corresponds to \href{https://oeis.org/A059712}{\texttt{OEIS A059712}}.
\end{corollary}

In \cite{LatticeAnimals}, Bousquet-Mélou and Rechnitzer derived the following generating functions for directed animals.

\begin{theorem}[Generating functions of directed animals {\cite[Proposition 1]{LatticeAnimals}}]\label{thm:gf_directed_animals}
  The generating function $Q(z)$ for strict half-pyramids is given by
  \[
    Q(z) = \frac{1 - z - \sqrt{(1+z)(1-3z)}}{2z}.
  \]
  The generating function for strict pyramids, with $z$ counting their number of dimers and $u$ counting their right width is
  \begin{equation} \label{eq:gf_pyramids}
    P(z,u) = \frac{Q(z)}{1 - uQ(z)}.
  \end{equation}
  In particular, the generating function $P(z,1)$ for directed animals on the square lattice is given by
  \[
    P(z,1) = \frac{1}{2}\left(\sqrt{\frac{1+z}{1-3z}}-1\right).
  \]
\end{theorem}

\begin{theorem}[Generating functions of stacked directed animals {\cite[Proposition 2]{LatticeAnimals}}]
  Let $Q(z)$ denote the generating function for strict half-pyramids. Let $P(z,u)$ denote the bivariate generating function for strict pyramids, with $u$ counting the right width of the pyramid. Then, the generating function for stacked directed animals on the square lattice with $z$ enumerating the number of dimers, $u$ the right width and $t$ the number of minimal dimers, is given by
  \[
    S(z,u,t) = \frac{t P(z,u)}{1 - t P(z,1)^2} 
    = \frac{t Q(1 - Q)^2}{(1 - uQ)((1- Q)^2 - tQ^2)}.
  \]
  In particular, the generating function for stacked directed animals on the square lattice, is given by 
  \[
    S(z) = \frac{(1-z)(1-3z) - (1-4z)\sqrt{(1-3z)(1+z)}}{2z(2 - 7z)}.
  \]
\end{theorem}

\begin{theorem}
  The generating function of stacked directed animals of size $n+1$ on the square lattice coincides with the generating function of Motzkin paths with catastrophes of length $n$.
\end{theorem}

\begin{proof}
  Let $E_{\mathcal{M}}(z)$ be the generating function of Motzkin excursions and $Q(z)$ be the generating function of strict half-pyramids.
  Then, comparing Theorem \ref{thm:gf_directed_animals} to \Cref{cor:motzkin_excursions} shows that 
  \begin{equation}\label{eq:Qs}
    Q(z) = z \cdot E_\mathcal{M}(z).
  \end{equation}
  Furthermore, for the bivariate generating function of strict pyramids $P(z,u)$, with $u$ marking the right width of the pyramid we have 
  $$
    P(z,u) = \frac{Q(z)}{1-uQ(z)} = \frac{z E_\mathcal{M}(z)}{1-uzE_\mathcal{M}(z)}.
  $$
  This generating function also has a lattice path interpretation.
  Let $\omega$ be a Motzkin excursion, with catastrophes only at altitude zero and let $u$ count the number of catastrophes in~$\omega$. We cut the path $\omega$ directly before each catastrophe. This partitions $\omega$ into a single Motzkin excursion without catastrophes, counted by $E_\mathcal{M}(z)$, followed by a possibly empty sequence of Motzkin excursions without catastrophes, each preceded by a catastrophe, counted by $u z \cdot E_\mathcal{M}(z)$.
  Hence, their generating function $F(z,u)$ satisfies
  \begin{equation}\label{eq:Ps}
    F(z,u) = \frac{E_\mathcal{M}(z)}{1-uzE_\mathcal{M}(z)} = \frac{P(z,u)}{z}.
  \end{equation}
  Further, the generating function for stacked directed animals reads 
  \begin{equation*}\label{eq:S}
    S(z,1,1) = \frac{P(z,1)}{1-P(z,1)^{2}} = \frac{\frac{Q(z)}{1 - Q(z)}}{1 - \frac{Q(z)^2}{(1-Q(z))^2}} = \frac{Q_s(z)}{1 - Q(z) - \frac{Q(z)^2}{1 - Q(z)}}
    = \frac{Q(z)}{1 - \frac{Q(z)}{1 - Q(z)}}.
  \end{equation*}
  Next, we observe that the generating function of Motzkin meanders satisfies
  \begin{equation}\label{eq:motzkin_meanders_animals}
    M_\mathcal{M}(z) = \frac{E_\mathcal{M}(z)}{1-zE_\mathcal{M}(z)}.
  \end{equation}
  To wit, consider a last passage decomposition of a Motzkin meander $\omega$. This splits $\omega$ into an initial excursion, counted by $E_{\mathcal{M}}(z)$, followed by a sequence of paths going from altitude $i$ to altitude $i + 1$, while staying above the line $y = i$, counted by $z E_\mathcal{M}(z)$.
  Finally, combining \eqref{eq:S}, \eqref{eq:Qs} and \eqref{eq:motzkin_meanders_animals}, we obtain
  \begin{equation*}
    S(z,1,1) = \frac{Q_s(z)}{1 - \frac{Q_s(z)}{1 - Q_s(z)}} = \frac{zE_\mathcal{M}(z)}{1 - z\frac{E_\mathcal{M}(z)}{1-zE_\mathcal{M}(z)}} = \frac{zE_\mathcal{M}(z)}{1 - zM_\mathcal{M}(z)} = zM_\mathcal{M}(z). \qedhere
  \end{equation*}
\end{proof}

\section{Bijective interpretation}
\label{sec:bijection}

\begin{figure}[hbt!]
  \centering
  \begin{subfigure}[c]{0.4 \textwidth}
    \centering
    \includegraphics[height=3cm]{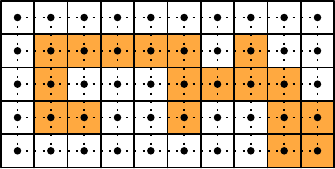}
    \caption{Stacked directed animal of size 18.}
  \end{subfigure}
  \begin{subfigure}[c]{0.59 \textwidth}
    \centering
    \includegraphics[height=5cm]{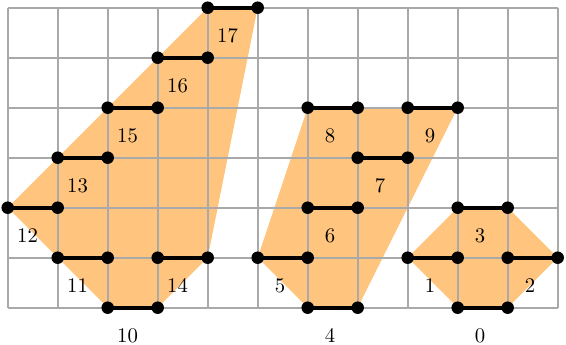}
    \caption{Corresponding stacked pyramids of size 18.}
  \end{subfigure}
  \vskip\baselineskip
  \begin{subfigure}{\textwidth}
    \centering
    \includegraphics{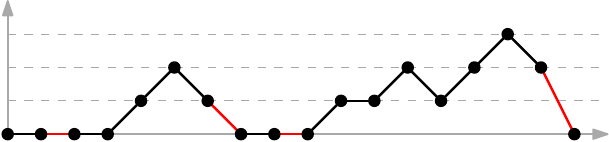}
    \caption{Motzkin excursion with catastrophes (marked in red) of length 17.}
  \end{subfigure}
  \caption[Bijection involving stacked directed animals.]{
  A stacked directed animal and its corresponding Motzkin excursion with catastrophes. 
  The dimers are numbered according to the order of their corresponding steps in the lattice path.}
  \label{fig:stacked_directed_animals_example}
\end{figure}

We are now ready to link Motzkin excursions with catastrophes to stacked directed animals.
The following bijection is introduced incrementally in three steps, starting with half-pyramids, leading up to pyramids and finally dealing with the entire class of stacked directed animals, which will prove the main theorem of this section.
We start with a bijection between strict half-pyramids and classical Motzkin paths, both of which are enumerated by the sequence \oeis{A001006}.

\begin{lemma} \label{lemma:half_pyramids}
  The set of strict half-pyramids of size $n+1$ is in bijection with the set of Motzkin excursions of length $n$.
\end{lemma}

\begin{proof}
  We already observed in \eqref{eq:Qs} that strict half-pyramids are counted by the Motzkin numbers. Now we will make the combinatorial origin of this connection explicit, by recursively constructing a bijection $\omega$ between these combinatorial classes. The recursive descriptions of both classes are pictured in Figure \ref{fig:half_pyramids}.

  Let $Q$ be a strict half-pyramid. It is either just a minimal dimer, or it consists of multiple dimers. In the first case, we set $\omega(Q)$ to be the empty path. 
  In the latter case, we further distinguish whether there is more than one dimer in the rightmost column of the half-pyramids. If there is just one, then $Q$ is just the product of its minimal dimer and a half-pyramid $Q'$ lying above the minimal dimer on its left side. In this case, we define $\omega(Q) := \textbf{E}\, \omega(Q')$.
  Otherwise, we push the lowest non-minimal dimer of the rightmost column upwards to obtain a factorization into the minimal dimer and two half-pyramids $Q_1$ and $Q_2$. This leads to the recursive rule $\omega(Q):= \textbf{NE}\,\omega(Q_1) \textbf{SE}\, \omega(Q_2)$.
  
  For the inverse direction, let $M$ be a Motzkin excursion. It is either just the empty walk or it consists of at least one step. In the first case, we set $\omega^{-1}(M)$ to be a single dimer.
  In the latter case, we further distinguish by the first step in $M$.
  If $M = \textbf{E}\, M'$, we place a single dimer on the $x$-axis and recursively build $\omega^{-1}(M')$ diagonally right above the minimal dimer.
  If otherwise $M$ starts with a \textbf{NE}-step, we identify the first \textbf{SE}-step that returns to the $x$-axis and partition $M = \textbf{NE}\, M_1 \textbf{SE}\, M_2$. Here we again start by placing a dimer on the $x$-axis and recursively building $\omega^{-1}(M_1)$ diagonally left above it. Once the construction of $\omega^{-1}(M_1)$ is complete, we place $\omega^{-1}(M_2)$ in the same column as the minimal dimer, diagonally right above $\omega^{-1}(M_1)$.
\end{proof}

\begin{figure}[hbt!]
  \centering
  \includegraphics[width = 0.8\linewidth, trim={0 0 0 1cm}, clip]{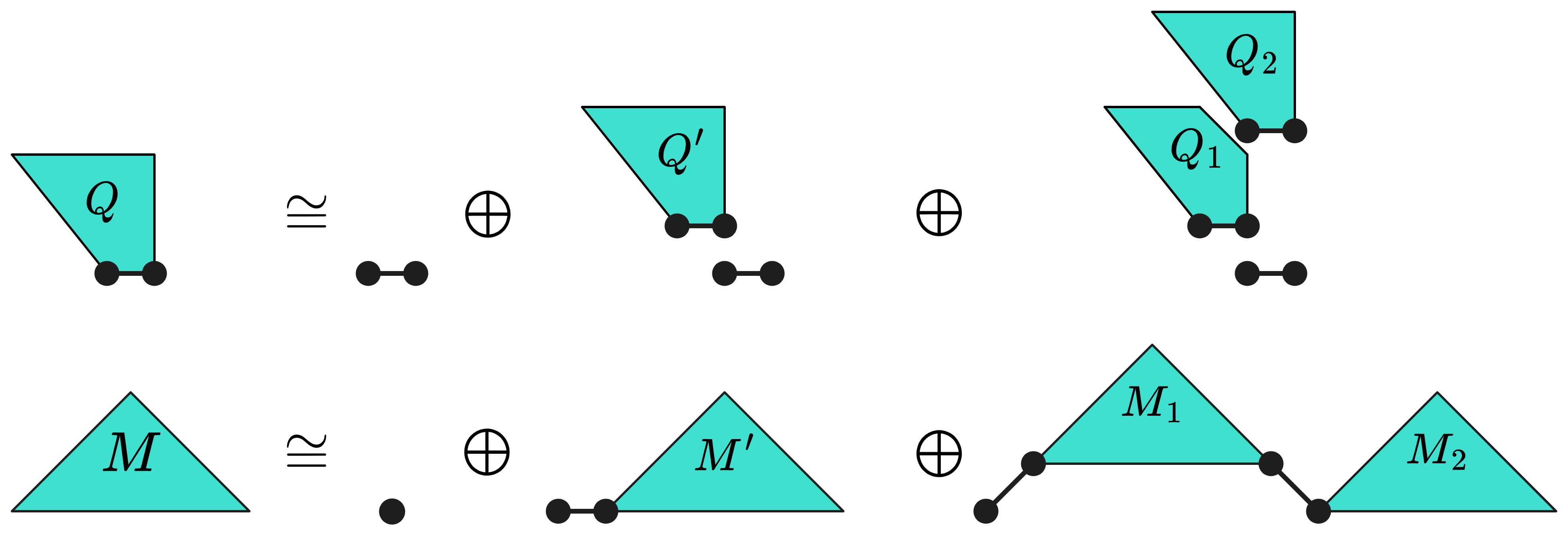}
  \caption{The factorizations of half-pyramids and Motzkin excursions.}
  \label{fig:half_pyramids}
\end{figure}

Building on this result, we present a bijection from strict pyramids to a subclass of Motzkin paths with catastrophes, where catastrophes are only allowed to occur at height $h = 0$.

\begin{lemma}\label{lemma:pyramids}
  The set of strict pyramids of size $n+1$ is in bijection with the set of Motzkin excursions of length $n$ with catastrophes only occurring at height $h = 0$.
\end{lemma}

\begin{proof}
  We already observed in \eqref{eq:Ps} that the generating functions of these two combinatorial classes coincide. Now we present a combinatorial argument for this fact, by extending the bijection $\omega$ from \Cref{lemma:half_pyramids} to the class of all pyramids.
  First, let $P$ be a strict pyramid. It either has zero right width and is thus a half-pyramid, or there exists a dimer exactly one step to the right of the minimal dimer at some height $h > 0$. In the first case, $\omega(P)$ is already well-defined due to Lemma \ref{lemma:half_pyramids}.
  In the second case, we partition $P$ into a lower half-pyramid $Q$ and an upper pyramid $P$, by pushing the lowest non-minimal dimer in the column of the minimal dimer upwards; see Figure~\ref{fig:pyramids}. In this case we apply the recursive rule $\omega(P) = \omega(Q) \ECat\, \omega(P')$, where we write \ECat\, to denote a (horizontal) catastrophe at height zero.

  For the reverse direction, consider a Motzkin excursion $M$ with catastrophes only at height $h = 0$. If it has no \ECat-step, it is simply a regular Motzkin excursion and Lemma \ref{lemma:half_pyramids} applies. In the other case, we split it at the first catastrophe into an initial Motzkin excursion, followed by a catastrophe and a final Motzkin excursion with catastrophes at height $h = 0$. 
  We apply Lemma \ref{lemma:half_pyramids} to the initial Motzkin excursion, which yields a half-pyramid. Then we drop a dimer in the column right of the minimal dimer and recurse on the part after the first catastrophe.
\end{proof}

\begin{figure}[hbt!]
  \centering
  \includegraphics[width = 0.8\linewidth, trim={0 0 0 12cm}, clip]{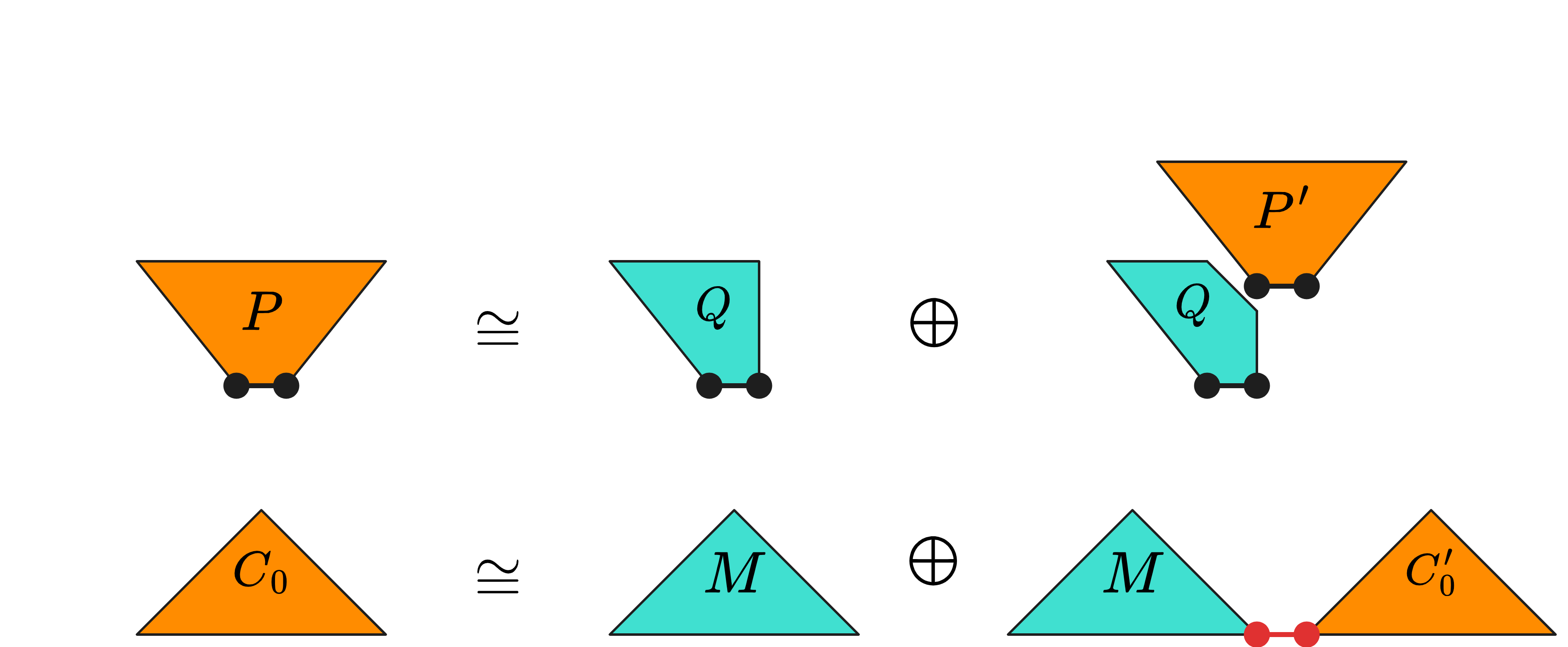}
  \caption[The factorization of Motzkin excursions with only horizontal catastrophes.]{The factorizations of strict pyramids and Motzkin excursions with only horizontal catastrophes.}
  \label{fig:pyramids}
\end{figure}

Finally, we are now able to present our main result.

\Bijection*

\begin{proof}
  Let $A$ be a stacked directed animal on the square lattice and $H = V(A)$ the corresponding connected heap of dimers according to the mapping described in \Cref{subsec:construction}.
  Further, let $P_1,P_2,\dots,P_k$ denote the corresponding pyramidal factors of $H$. We start our translation into lattice paths with the rightmost pyramid $P_k$. If $k = 1$, we simply apply Lemma \ref{lemma:pyramids} to translate $H$ into a Motzkin excursion with catastrophes only occurring at height $h = 0$.
Otherwise, if $k > 1$, after we have drawn $P_k$,
the so far unused catastrophes from heights $h > 0$ will now encode the distance between the two pyramids $P_k$ and $P_{k-1}$, which we define as the horizontal distance between the leftmost dimer of $P_k$ and the minimal dimer of $P_{k-1}$.
Let us denote this distance with $\ell$, which will correspond to the height of the following catastrophe. 
Note that the right width of $P_{k-1}$ must be at least $\ell$, since otherwise the horizontal projections of $P_k$ and $P_{k-1}$ would not intersect. To compensate for the height loss due to the catastrophe, we treat the first $\ell$ half-pyramids $Q_{k-1,1},\dots,Q_{k-1,\ell}$, where the minimal dimer of each $Q_{k-1,i}$ is dropped in the column to the right of the minimal dimer of $Q_{k-1,i-1}$ differently. Let $P_{k-1}^\prime$ denote the pyramid placed to the right of the minimal dimer of $Q_{k-1,\ell}$ and note that its minimal dimer is the first dimer whose horizontal projection intersects with the horizontal projection of $P_k$, thus connecting the pyramids.
Now we need to deviate from the construction presented in Lemma \ref{lemma:pyramids}, as we need to introduce $\ell$ additional \textbf{NE}-steps in order to offset the height lost with the new catastrophe. 
Hence, the start of each of the half-pyramids $Q_i$ will be marked with a \textbf{NE}-step instead of with a horizontal catastrophe, like in Lemma \ref{lemma:pyramids}. 
In particular, this means that the start of a new pyramid is always marked with an additional \textbf{NE}-step. 
This additional step is important, as otherwise each pyramid consisting of $m$ dimers would be translated to a lattice path of length $m - 1$, and the final length of the lattice path would depend on the number of pyramids.
The half-pyramids themselves are then simply translated according to the recursion rules from Lemma \ref{lemma:half_pyramids}. 
Note that these rules remain legitimate on altitude $i > 0$, as they do not involve horizontal catastrophes, which may only happen at height $0$.
Thus, the last half-pyramid $Q_{k-1,\ell}$ before $P_{k-1}'$ will be represented by a Motzkin excursion starting and ending at height $\ell$. 
After that, a catastrophe from height $\ell$ will usher in the start of the image of $P_{k-1}'$, which can now again be drawn according to the rules of Lemma \ref{lemma:pyramids}, as it no longer starts at a positive height.
This procedure, illustrated in Figure \ref{fig:stacked_pyramids}, can now be iterated over all pyramidal factors of $H$ to obtain the final lattice path image of $H$.

For the inverse mapping, let $M$ be a Motzkin excursion with catastrophes. If $M$ does not contain any non-horizontal catastrophes, we may simply apply Lemma \ref{lemma:pyramids} to translate $M$ to a single pyramid. Otherwise, we split $M$ at every non-horizontal catastrophe. This yields a set of excursions $E_1,E_2,\dots,E_k$, with $k > 1$, each having exactly one non-horizontal catastrophe at their very end.
Consider the first of these excursions $E_1$, which will correspond to the rightmost pyramid $P_k$ of $H$ and the start of the next pyramid $P_{k-1}$. To recover $P_k$, it suffices to apply the procedure described in Lemma \ref{lemma:pyramids}.
However, this alone does not yet tell us, at which point we need to start drawing $P_{k-1}$. For that we need to look ahead to the non-horizontal catastrophe, which signals the end of $E_1$. The start of $P_{k-1}$ then corresponds to the last time $E_1$ leaves altitude zero before its final catastrophe, which can be intuitively described as the first \textbf{NE}-step visible from the viewpoint of the next catastrophe. The next question we need to answer is where to place the minimal dimer of $P_{k-1}$. 
For this we start at the horizontal projection of the leftmost dimer of $P_k$ and move $\ell + 1$ units to the left, where $\ell$ is the height of the catastrophe at the end of $E_1$. This is where we place the minimal dimer of $P_{k-1}$ and start building the first half-pyramid $Q_{k-1,1}$.
Similarly, the last time $E_1$ leaves altitude one marks the start of the next half-pyramid $Q_{k-1,i+1}$. The minimal dimer of $Q_{k-1,i+1}$ needs to be placed diagonally right above the highest dimer in the rightmost column of $Q_{k-1,i}$.
Now we can iterate this process until we hit the catastrophe, which marks the start of the pyramid $P_{k-1}'$. 
Now the process repeats, as we draw $P_{k-1}'$ until we reach the first \textbf{NE}-step visible from the next non-horizontal catastrophe; see Figure \ref{fig:stacked_directed_animals_example} for an example of this correspondence.
\end{proof}

\section{Improved bounds on the width of stacked pyramids}
\label{sec:parameters}

One advantage gained by viewing stacked directed animals as lattice paths comes with the reinterpretation of parameters in the language of lattice paths. In particular, studying the asymptotic behaviour of the cumulative size of all catastrophes improves the previous best asymptotic lower bound of $\frac{3}{28}n$ \cite[Proposition 2]{LatticeAnimals} by a factor of three.

\begin{lemma}
  \label{lem:cum_size}
    Define the random variable $S_n$ as the cumulative size of all catastrophes of a Motzkin excursion with catastrophes of length $n$ drawn uniformly at random.
    Then, $\lim_{ n \to \infty } \mathbb{E}[S_n] = \frac{3}{14}n$.
\end{lemma}

\begin{proof}
  A key result \cite{BasicLatticePaths} in the enumeration of lattice paths is the fact that the principal small root $u_1(z)$ and the principal large root $v_1(z)$ of the kernel equation $1 - z \cdot P(u) = 0$ are conjugated to each other at their dominant singularity $\rho = 1 / P(\tau)$, where $\tau$ is the minimal positive real solution to $P'(\tau) = 0$. This singularity also plays a crucial role in determining the limit laws of various parameters of directed lattice paths.
  Further, we also need to consider the dominant singularity of the generating function
  \[
    D(z) = \frac{1}{1 - Q(z,1)} = \frac{1}{1 - z \cdot M(z,1)} = \frac{1 - u_1(z)}{1 - z \cdot P(1)}
  \]
  of Motzkin excursions ending with a catastrophe.
  Its dominant singularity, which we denote by $\rho_0$ is given by the minimum of the minimal real positive solution of $1 - Q(z,1) = 0$ and $\rho$.
  For the jump polynomial of Motzkin paths $P(u) = 1/u + 1 + u$, we get that $\rho = 1$ and $\rho_0 = 2/7$.
  For the average cumulative size of all catastrophes we apply \cite[Theorem~4.18]{Catastrophes}, which tells us that the normalized random variable $\frac{S_n - \mu n}{\sigma \cdot \sqrt{ n }}$ converges in law to a standard Gaussian variable $X \sim \mathcal{N}(0,1)$.
  The parameter $\mu$ is given by $\mu = \frac{Q_u(\rho_0, 1)}{\rho_0 \cdot Q_z(\rho_0, 1)}$, which works out to be $3/14$.
  Thus, we get $\mathbb{E}[S_n] = \frac{3}{14} n$.
\end{proof}

In order to obtain a non-trivial upper bound on the average width of a stacked pyramid, we look at the asymptotic number of \textbf{SE} steps in a Motzkin excursion with catastrophes.

\begin{lemma}
  \label{lem:no-SE}
  Let $X_n$ denote the random variable counting the number of \textbf{SE} steps in a Motzkin excursion of length $n$ chosen uniformly at random.
  Then, $\mathbb{E}[X_n] \xrightarrow[]{{n \to \infty}} \frac{n}{7}$.
\end{lemma}

\begin{proof}
  We mark the \textbf{SE} steps in a Motzkin path by introducing the bivariate jump polynomial $P(u,x) := \frac{x}{u} + 1 + u$. 
  This leads to the bivariate generating function
  \[
  E_\mathcal{M}(z,x) = \frac{u_1(z,x) \left(-1+u_1(z,x) \right)}{z \left(-1+\left(x -u_1(z,x) +3\right) z \right) x},
  \]
  of Motzkin excursion with catastrophes, where $u_1(z,x)$ denotes the small branch of the kernel equation $1 - z \cdot P(u,x) = 0$. The variable $x$ counts the number of \textbf{SE} steps.
   Then, 
  \[
    \mathbb{P}[X_n = k] = \frac{[z^n x^k] E_\mathcal{M}(z,x)}{[z^n] E_\mathcal{M}(z,1)} \quad \text{and} \quad \mathbb{E}[X_n] = [z^n] \left(\frac{\frac{\partial}{\partial x} E_\mathcal{M}(z,1)}{ E_\mathcal{M}(z,1)}\right).
  \]
  With the help of a computer algebra system one obtains that
  \[
    [z^n] \frac{\partial}{\partial x} E_\mathcal{M}(z,1) = \frac{3n}{56} \cdot \left( \frac{7}{2}\right)^n \cdot \left(1 + \mathcal{O}\left(\frac{1}{n}\right)\right)  \quad \text{and} \quad [z^n] E_\mathcal{M}(z,1) = \frac{3}{8} \cdot \left( \frac{7}{2}\right)^n \cdot \left(1 + \mathcal{O}\left(\frac{1}{n}\right)\right)
  \]
  and thus $\mathbb{E}[X_n] \xrightarrow[]{{n \to \infty}} \frac{n}{7}$.
\end{proof}

\Width*

\begin{proof}
  For the lower bound on the average width we recall that height of a non-horizontal catastrophe corresponds to the horizontal distance between the minimal dimer and the left-most dimer of two adjacent pyramids. This in turn is a lower bound on the distance between the minimal dimers of two adjacent pyramids.
  From \Cref{lem:cum_size} we know that the average cumulative size of a Motzkin excursion with catastrophes of size $n$ converges to $\frac{3}{14} \cdot n$.
  Further, the number of minimal dimers of a stacked pyramid is asymptotically equal to $\frac{3}{28} \cdot n$ according to \cite[Proposition 2]{LatticeAnimals}.
  Adding these numbers up, we arrive at our lower bound of $\frac{9}{28} \cdot n$.

  For the upper bound, we observe that a \textbf{SE} step in a Motzkin excursion with catastrophes never increases the width of the corresponding pyramid. This is because if we identify steps with dimers, the dimer corresponding to a \textbf{SE} step is always placed in a column that already contains at least one other dimer.
  Thus, \Cref{lem:no-SE} shows that the asymptotic width of stacked pyramids is upper bounded by $\frac{6}{7} \cdot n$.
\end{proof}

%%%%%%%%%%%%%%%%%%%%%%%%%%%%%%%%%%%%%%%%%%%%%%
\section*{Acknowledgments}

Florian Schager was supported by the Austrian Science Fund (FWF): P~36280. \\
Michael Wallner was supported by the Austrian Science Fund (FWF): P~34142. 

%%%%%%%%%%%%%%%%%%%%%%%%%%%%%%%%%%%%%%%%%%%%
\bibliographystyle{eptcs}
\bibliography{literature}    

\end{document}